\documentclass[reqno]{amsart}

\usepackage[latin1]{inputenc}
\usepackage[T1]{fontenc}                  
\usepackage[english]{babel}
\usepackage{amsmath,amssymb,amsbsy,amsthm,amsfonts}
\usepackage{color}
\usepackage{epsfig}
\usepackage{pgf}

\newtheorem{definition}{Definition}[section]
\newtheorem{example}[definition]{Example}
\newtheorem{remark}[definition]{Remark}

\newtheorem{theorem}[definition]{Theorem}
\newtheorem{proposition}[definition]{Proposition}

\newtheorem{lemma}[definition]{Lemma}

\usepackage[colorlinks=true]{hyperref}
\usepackage{theoremref}
\newcommand{\bluee}{\textcolor{black}}

\usepackage{accents}

\usepackage{wasysym}

\usepackage{enumerate}
\usepackage{xcolor}

\usepackage{booktabs}

\usepackage[makeroom]{cancel}

\usepackage{xcolor} 
\definecolor{tssteelblue}{RGB}{70,130,180}
\definecolor{tsorange}{RGB}{255,138,88}
\definecolor{ocre}{RGB}{52,177,201} 

\newcommand{\tssteelblue}[1]{{\color{tssteelblue}#1}}

\definecolor{coral}{RGB}{240,128,128} 

\definecolor{salmon}{RGB}{255,160,122} 

\definecolor{lightpink}{RGB}{255,182,193} 

\definecolor{pink}{RGB}{255,192,203} 

\definecolor{palevioletred}{RGB}{219,112,147} 

\definecolor{deeppink}{RGB}{255,20,147} 
\newcommand{\deeppink}[1]{{\color{deeppink}#1}}

\definecolor{violet}{RGB}{238,130,238} 

\definecolor{mediumorchid}{RGB}{186,85,211} 

\definecolor{darkmagenta}{RGB}{139,0,139} 

\definecolor{indianred}{RGB}{205,92,92} 

\definecolor{maroon}{RGB}{128,0,0} 

\definecolor{darkcyan}{RGB}{0,139,139} 

\definecolor{gold}{RGB}{255,215,0} 

\begin{document}
\title[Generalization of Completely Monotone Conjecture for Tsallis entropy] 
      {Generalization of Completely Monotone Conjecture for Tsallis entropy}  
\author[L.-C. Hung]{Li-Chang Hung$^{\ast}$}

\email{lichang.hung@gmail.com}
\thanks{$^{\ast}$Department of Civil Engineering, National Taiwan University, Taipei, Taiwan}


\begin{abstract}
We generalize the completely monotone conjecture (\cite{Cheng-Higher-order-derivatives-Costa-IEEE-2015}) from Shannon entropy to the Tsallis entropy up for orders up to at least four. To this end, we employ the algorithm (\cite{jungel2016entropy,jungel2006algorithmic}) which employs the technique of systematic integrations-by-parts.

\end{abstract}

\maketitle

\section{Introduction}\label{sec: intro}

Suppose that $X$ is a continuous random variable with probability density function $g(x)$. Given $\alpha>0$, we introduce the entropy defined  by
\begin{equation}\label{eqn: Tsallis entropy generalized with IC g}
H_{\alpha}[X]=\frac{1}{\alpha-1}\left(\int_{\mathbb{R}} g(x)\,dx-\int_{\mathbb{R}} 
\left(
\frac{1}{\sqrt{2\,\pi\,t}}\int_{-\infty}^{\infty} g(y)\,e^{-\frac{(y-x)^2}{2\,t}}\,dy
\right)
^{\alpha}
\,dx\right).
\end{equation}
By reversing the order of integration, it is readily seen that
\begin{equation}
\int_{\mathbb{R}} g(x)\,dx=
\int_{\mathbb{R}} 
\left(
\frac{1}{\sqrt{2\,\pi\,t}}\int_{-\infty}^{\infty} g(y)\,e^{-\frac{(y-x)^2}{2\,t}}\,dy
\right)
\,dx
\end{equation}
holds. Applying L'H\^{o}pital's rule gives
\begin{equation}
\lim_{\alpha\to1}H_{\alpha}[X]=-\int_{\mathbb{R}} u\,\log u\,dx:=\mathcal{S}[u],
\end{equation}
where $\mathcal{S}[u]$ is the Shannon entropy (\cite{shannon1948mathematical-communication}) and $u=u(x,t)$ is given by
\begin{equation}\label{eqn: convolution formula for Cauchy problem of heat equation}
u(x,t)=\frac{1}{\sqrt{2\,\pi\,t}}\int_{-\infty}^{\infty} g(y)\,e^{-\frac{(y-x)^2}{2\,t}}\,dy.
\end{equation}
Now we can rewrite $H_{\alpha}[X]$ as 
\begin{equation}\label{eqn: Tsallis Entropy with integral of g}
H_{\alpha}[X]=\frac{1}{\alpha-1}
\left(
\int_{\mathbb{R}} g(x)
\,dx
-\int_{\mathbb{R}} 
\left(
u(x,t)
\right)
^{\alpha}
\,dx\right),
\end{equation}
which we call a generalized Tsallis entropy as when $\int_{\mathbb{R}} g(x)
\,dx=1$, $H_{\alpha}[X]$ in \eqref{eqn: Tsallis Entropy with integral of g} reduces to the Tsallis entropy (\cite{tsallis1988possible}). 

\vspace{5mm}

\begin{definition}[Tsallis Entropy]\thlabel{def: Tsallis Entropy}
 Let $w=w(x)$ be a probability density function. The Tsallis entropy is defined by
\begin{equation}\label{eqn: Tsallis Entropy}
H_{\alpha}[w]=\frac{1}{\alpha-1}\left(1-\int_{\mathbb{R}} w^{\alpha}\,dx\right),
\end{equation}
where $\alpha\in(0,1)\cup(1,\infty)$ is a constant.

\end{definition}

\vspace{5mm}

The Tsallis entropy introduced in \thref{def: Tsallis Entropy} generalizes Shannon entropy in the sense of 
\begin{equation}
\frac{1}{\alpha-1}\left(1-\int_{\mathbb{R}} w^{\alpha}\,dx\right) 
\to 
-\int_{\mathbb{R}} w\,\log w\,dx
\end{equation}
as $\alpha\to1$.

From the viewpoint of partial differential equations (PDEs), \eqref{eqn: convolution formula for Cauchy problem of heat equation} can be interpreted as the solution formula to the Cauchy problem for the heat equation, i.e. $u=u(x,t)$ given by \eqref{eqn: convolution formula for Cauchy problem of heat equation} satisfies
\begin{equation}\label{eqn: Cauchy problem for heat eqn}
\begin{cases}
\vspace{3mm} 
\dfrac{\partial u}{\partial t}=\dfrac{1}{2}\,\dfrac{\partial^2 u}{\partial x^2},\ \ &(x,t)\in (-\infty,\infty)\times(0,\infty),\\
u(x,0)=g(x),\ \ &x\in (-\infty,\infty),
\end{cases}
\end{equation}
where $g\in C(\mathbb{R})\bigcap L^{\infty}(\mathbb{R})$ is the initial condition. Also, \eqref{eqn: convolution formula for Cauchy problem of heat equation} can be viewed as a convolution formula
\begin{equation}
u(x,t)=(g\ast \phi)(x,t)
=
\int_{-\infty}^{\infty} g(y)\,\phi(x-y,t)\,dy
=
\frac{1}{\sqrt{2\,\pi\,t}}\int_{-\infty}^{\infty} g(y)\,e^{-\frac{(y-x)^2}{2\,t}}\,dy,
\end{equation}
where $\phi(x,t)=\frac{1}{\sqrt{2\,\pi\,t}}\,e^{-\frac{x^2}{2\,t}}$ is the fundamental solution to the heat equation $u_t=\frac{1}{2}u_{xx}$. Henceforth, we assume that $g\in C(\mathbb{R})\bigcap L^{\infty}(\mathbb{R})$ is a probability density function, unless otherwise stated. 

\vspace{5mm}

\begin{proposition}[Theorem 2.1 in \cite{dibenedetto2009PDE}]\thlabel{thm: dibenedetto's convolution formula for the Cauchy problem heat eqn}
Let $g\in C(\mathbb{R})\bigcap L^{\infty}(\mathbb{R})$. Then
\begin{equation}
u(x,t)=\frac{1}{\sqrt{2\,\pi\,t}}\int_{-\infty}^{\infty} g(y)\,e^{-\frac{(y-x)^2}{2\,t}}\,dy.
\end{equation}
is the only bounded solution to the Cauchy problem \eqref{eqn: Cauchy problem for heat eqn}.
\end{proposition}

\vspace{5mm}

 As a consequence of \thref{thm: dibenedetto's convolution formula for the Cauchy problem heat eqn}, the condition $g\in C(\mathbb{R})\bigcap L^{\infty}(\mathbb{R})$ ensures us to use the heat equation $u_t=\frac{1}{2}u_{xx}$ and the convolution formula \eqref{eqn: convolution formula for Cauchy problem of heat equation} interchangeably, as long as we only consider bounded solution to the Cauchy problem \eqref{eqn: Cauchy problem for heat eqn}. 

\vspace{5mm}


In this paper, our main purpose is to investigate the \textit{completely monotone conjecture} proposed in \cite{Cheng-Higher-order-derivatives-Costa-IEEE-2015}. We shall show in the later sections that Theorem 1., Theorem 2., and Lemma 1. in \cite{Cheng-Higher-order-derivatives-Costa-IEEE-2015} can be generalized to the Tsallis entropy as in \thref{thm: CMC conjecture for n=1 to 4}.

\vspace{5mm}

\begin{theorem}[Completely monotone property for the Tsallis entropy along the heat flow]\thlabel{thm: CMC conjecture for n=1 to 4}
Suppose that $u(x,t)$ is the bounded solution to the Cauchy problem \eqref{eqn: Cauchy problem for heat eqn}. Then for $n=1,2,3,4,5,$
\begin{equation}\label{eqn: CMC conjecture for n=1,2,3}
(-1)^{n+1}\,\dfrac{d^n}{dt^n}H_{\alpha}[u]\ge0 \quad \forall t>0
\end{equation}
holds for those $\alpha's$ given in Table~\ref{tab: n, Tsallis, alpha}.
\end{theorem}

\begin{table}[ht]   
\caption{$\alpha$ for \thref{thm: CMC conjecture for n=1 to 4}}   
\label{tab: n, Tsallis, alpha}   
\centering   
\begin{tabular}{p{1cm}p{2cm}p{4.0cm}p{3.0cm}} 
\toprule    & $\frac{d^n}{dt^n}H_{\alpha}[u(t)]$ &interval of $\alpha$ & degree of polymonial
\\   
\midrule   
$n=1$         &$\ge0$	&$(0,1)\cup(1,\infty)$ & quadratic ($2\,n=2$) 	
\\   
$n=2$         &$\le0$	&$(0,1)\cup(1,3]$ & quartic ($2\,n=4$)	
\\   
$n=3$         &$\ge0$	&$(\alpha_0,1)\cup(1,2]$, $\alpha_0\approx0.389$ & sextic ($2\,n=6$)	
\\  
$n=4$         &$\le0$	&$(1,2)$ & octic ($2\,n=8$)
\\   
$n=5$         &$\ge0$	&$(\alpha_1,2)$, $\alpha_1\in(1.54,1.55)$ & decic ($2\,n=10$)
\\      
\bottomrule   
\end{tabular} 
\end{table}

Our approach to prove \thref{thm: CMC conjecture for n=1 to 4} is to use the algorithm for the construction of entropies developed in \cite{jungel2016entropy,jungel2006algorithmic}.

\begin{itemize}
  \item (\textbf{Systematic integrations-by-parts})  As in \cite{Cheng-Higher-order-derivatives-Costa-IEEE-2015}, the technique of integrations-by-parts is employed (\cite{jungel2016entropy,jungel2006algorithmic}). The main difference is that the algorithm uses more systematic integrations-by-parts, and leads to a decision problem for polynomial systems (Section~\ref{sec: Algorithmic construction of entropies}).  
  \item (\textbf{Decic polynomial}) Due to complicated calculation, the algorithm in \cite{jungel2016entropy,jungel2006algorithmic} is used to the polynomials up to degree 6 (i.e. quadratic, quartic, and sextic polynomials). With the aid of Mathematica computer algebra system, we can use the algorithm to the polynomials up to 10 (i.e., octic and decic polynomials). Each time derivative $\frac{d^n}{dt^n}H_{\alpha}[u(t)]$ corresponds to a polynomial of degree $2\,n$ (Table~\ref{tab: n, Tsallis, alpha}).
  \item (\textbf{Generalization from Shannon entropy to the Tsallis entropy}) Table~\ref{tab: n, Tsallis, alpha} shows that, for $n=1,2,3,4$ we can let $\alpha\to1$ in \thref{thm: CMC conjecture for n=1 to 4}, and thus Theorem 1., Theorem 2., and Lemma 1. in \cite{Cheng-Higher-order-derivatives-Costa-IEEE-2015} are all recovered. Unfortunately, when $n=5$, interval of $\alpha$ in Table~\ref{tab: n, Tsallis, alpha} does not allow us to pass to the limit $\alpha\to1$.  
\end{itemize}

For the fifth-order time derivative $\frac{d^5}{dt^5}H_{\alpha}[u(t)]$, the completely monotone conjecture has been shown to hold when the probability density function $g$ is log-concave (\cite{zhang2018gaussian-log-concave}). For a brief review of the completely monotone conjecture and other related conjectures, we refer to \cite{ledoux2022CMC-summary}.

\vspace{5mm}
\section{First-order time derivative}
\vspace{5mm}

\begin{lemma}\thlabel{thm: ux/u^k=0 as |x|tends to infty}
Let $u(x,t)$ be the bounded solution to the Cauchy problem \eqref{eqn: Cauchy problem for heat eqn}. For any $0<k<1$,
\begin{equation}
\lim_{|x|\to\infty}\frac{u_x(x,t)}{(u(x,t))^k}=0, \quad t>0.
\end{equation}
\end{lemma}

\begin{proof}
\ \

 Applying Proposition 2. in \cite{Cheng-Higher-order-derivatives-Costa-IEEE-2015} leads to $u\cdot\left(\frac{u_x}{u}\right)^{\hat{k}}(\pm\infty,t)=0$ for any $\hat{k}\in\mathbb{Z}_{+}$. We rewrite $u\cdot\left(\frac{u_x}{u}\right)^{\hat{k}}(\pm\infty,t)=0$ as $\left(\frac{u_x}{u^{\frac{\hat{k}-1}{k}}}\right)^{\hat{k}}(\pm\infty,t)=0$ and thus find 
\begin{equation}\label{eqn: u_x/u^k=0 at infinity}
\frac{u_x}{u^{\frac{\hat{k}-1}{\hat{k}}}}(\pm\infty,t)=0 
\end{equation} 
for any $\hat{k}\in\mathbb{Z}_{+}$. We prove the desired result as follows:
\begin{align}\nonumber
&\forall 0<k<1,\quad \exists \hat{k}=\left[\frac{1}{1-k}\right]+1\in\mathbb{Z_{+}}
&& \text{\tssteelblue{$\frac{\hat{k}-1}{\hat{k}}>k$ $\iff$ $\hat{k}-1>k\,\hat{k}$ $\iff$}} 
\\[1ex] \notag
&\text{such that}\; \frac{\hat{k}-1}{\hat{k}}>k
&& \text{\tssteelblue{$\hat{k}\,(1-k)>1$ $\iff$ $\hat{k}>\frac{1}{1-k}$. Let $\hat{k}=\left[\frac{1}{1-k}\right]+1\in\mathbb{Z}_{+}$}} 
\\[1ex] \notag
&\implies\frac{u_x}{u^{k}}(\pm\infty,t)=\frac{u_x}{u^{\frac{\hat{k}-1}{\hat{k}}}}\cdot u^{\frac{\hat{k}-1}{\hat{k}}-k}(\pm\infty,t)=0
&& \text{\tssteelblue{$\hat{k}\in\mathbb{Z_{+}}$ $\implies$ $\frac{u_x}{u^{\frac{\hat{k}-1}{\hat{k}}}}(\pm\infty,t)=0$ using \eqref{eqn: u_x/u^k=0 at infinity}}} 
\\[1ex] \notag
&
&& \text{\tssteelblue{$\frac{\hat{k}-1}{\hat{k}}-k>0$ and $\lim_{|x|\to\infty} u(x,t)=0$ $\implies$ $u^{\frac{\hat{k}-1}{\hat{k}}-k}(\pm\infty,t)=0$}}
\end{align}

\end{proof}

\vspace{5mm}

\begin{theorem}[First-order time derivative]\thlabel{thm: First Time Derivative}
Let $\alpha\in(0,1)\cup(1,\infty)$. Suppose that $u(x,t)$ is the bounded solution to the Cauchy problem \eqref{eqn: Cauchy problem for heat eqn}. Then for $t>0$,
\begin{equation}\label{eqn: 1st time derivative}
\dfrac{d}{dt}H_{\alpha}[u(t)]
=
\frac{\alpha}{2}
\int_{\mathbb{R}}
u^{\alpha-2}\,
(u_x)^2\,
dx\ge0.
\end{equation}
\end{theorem}

\begin{proof}

\begin{align}\nonumber
&\dfrac{d}{dt}H_{\alpha}[u(t)]
=
-\frac{\alpha}{\alpha-1}
\int_{\mathbb{R}}
u^{\alpha-1}\,
u_t\,
dx
&& \text{\tssteelblue{\thref{def: Tsallis Entropy}}} 
\\[1ex] \notag
&=
-\frac{\alpha}{2\,(\alpha-1)}
\int_{\mathbb{R}}
u^{\alpha-1}\,
u_{xx}\,
dx
&& \text{\tssteelblue{$u_t=\frac{1}{2}u_{xx}$}} 
\\[1ex] \notag
&=
-\frac{\alpha}{2\,(\alpha-1)}
\int_{\mathbb{R}}
u^{\alpha-1}\,
du_x\,
&& \text{\tssteelblue{$u_{xx}dx=du_x$}} 
\\[1ex] \notag
&= 
-\frac{\alpha}{2\,\cancel{(\alpha-1)}}
\left(
u^{\alpha-1}\,
u_x\bigg|_{-\infty}^{\infty}
-\cancel{(\alpha-1)}
\int_{\mathbb{R}}
u^{\alpha-2}\,
(u_x)^2
\,dx
\right)
&& \text{\tssteelblue{integration by parts; $u_x(\pm\infty,t)=u(\pm\infty,t)=0$}}
\\[1ex] \notag
&= 
\frac{\alpha}{2}
\int_{\mathbb{R}}
u^{\alpha-2}\,
(u_x)^2\,
dx.
&& \text{\tssteelblue{Lemma~\ref{thm: ux/u^k=0 as |x|tends to infty}: $0<\alpha<1$ $\implies$ $\lim_{|x|\to\infty}\frac{u_x(x,t)}{(u(x,t))^{1-\alpha}}=0$}}
\end{align}

\end{proof}

\vspace{5mm}
\section{Algorithmic construction of entropies}\label{sec: Algorithmic construction of entropies}
\vspace{5mm}

To find the other time derivatives of the Tsallis entropy, we introduce the algorithm slightly modified from that proposed in \cite{jungel2016entropy,Jungel-algorithmic-construction-of-entropies-Nonlinerity-2006}. Before doing that, let us list some preliminaries.


\vspace{5mm}

\begin{theorem}[Fa\`a di Bruno's formula]\thlabel{thm: Faa di Bruno's formula use to find general formula for S0}
\ \
\begin{equation}\label{eqn: Faa di Bruno's formula}
\frac{d^m}{dt^m} g(f(t))
=
\sum \frac{m!}{b_1!b_2!\cdots b_m!}
g^{(k)}(f(t))
\left(
\frac{f'(t)}{1!}
\right)^{b_1} 
\left(
\frac{f''(t)}{2!}
\right)^{b_2} 
\cdots
\left(
\frac{f^{(m)}(t)}{m!}
\right)^{b_m}, 
\end{equation}
where the sum is over all different solutions in nonnegative integers $b_1,\dotsc, b_m$ of $b_1+2\,b_2+3\,b_3+\cdots+m\,b_m=m$ and $k:=b_1+b_2+b_3+\cdots+b_m$.
\end{theorem}

\vspace{5mm}

For the convenience of notation, let us introduce $\xi:=(\xi_1,\xi_2,\xi_3,\cdots)$, where $\xi_i$ is defined by
\begin{equation}
\xi_i:=\dfrac{1}{u}\cdot\dfrac{\partial^{i} u}{\partial x^i}, \quad i\in\mathbb{N}.
\end{equation}
The integration by parts formulas play an important role in the algorithm.

\vspace{5mm}

\begin{lemma}[Integration by parts formulas for $n$-th time derivative]\thlabel{lem: general formula for ibp in algorithm}
Suppose that $u=u(x,t)$ is the solution to the Cauchy problem \eqref{eqn: Cauchy problem for heat eqn}. For $j=1,2,\dotsc, r$, let $\mathcal{P}_j=(p_{j,1},p_{j,2},\dotsc,p_{j,2\,n-1})$ be a solution of 
\begin{equation}\label{eqn: eqn for sum of ip=2m-1}
\sum_{i=1}^{2\,n-1}i \cdot p_{j,i}=2\,n-1, \quad p_{j,i}\in \mathbb{N}\cup\{0\},
\end{equation}
where $r$ is the number of solutions of \eqref{eqn: eqn for sum of ip=2m-1}. Let $\xi=(\xi_1,\xi_2,\dotsc,\xi_{2\,n-1})\in\mathbb{R}^{2\,n-1}$ and
\begin{equation}
T_j(\xi):=
\frac{1}{u^{\alpha}}
\frac{\partial}{\partial x}
\left(
u^{\alpha}
\prod_{i=1}^{2\,n-1}
\xi_i^{p_{j,i}}
\right),
\quad j=1,2,\dotsc, r.
\end{equation}
Then 
\begin{equation}
\int_{\mathbb{R}}u^{\alpha}\,T_j(\xi)\,dx=0, \quad j=1,2,\dotsc,r.
\end{equation}

\end{lemma}

\vspace{5mm}

We are now in the position to introduce our algorithm for calculating $n$-th time derivative of $H_{\alpha}[u(t)]$ as follows:

\begin{enumerate}[(1)]
  \item (Scaling) By letting $t=2\hat{t}$ or $\frac{t}{2}=\hat{t}$, the heat equation $u_t=\frac{1}{2}u_{xx}$ in \eqref{eqn: Cauchy problem for heat eqn} is rescaled to $u_{\hat{t}}=u_{xx}$. It is readily seen that
 
\begin{equation}\label{}
\dfrac{d^n}{dt^n}H_{\alpha}[u(t)]
=
\frac{1}{2^n}\dfrac{d^n}{d\hat{t}^n}H_{\alpha}[u(2\hat{t})].
\end{equation}

  \item (Find all the possible $\displaystyle\prod_{i=1}^{2\,n}\xi_i^{p_{j,i}}$) For $j=1,2,\dotsc,\ell$, let $\mathcal{P}_j=(p_{j,1},p_{j,2},\dotsc,p_{j,2\,n})$ be a solution of 
\begin{equation}
\sum_{i=1}^{2\,n}i \cdot p_{j,i}=2\,n, \quad p_{j,i}\in \mathbb{N}\cup\{0\},
\end{equation}
where $\ell$ is the number of solutions. We find in the problem of finding $n$-th time derivative \textit{all the possible} $\mathcal{Q}_j(\xi)$ given by
\begin{equation}
\mathcal{Q}_j(\xi):=
\prod_{i=1}^{2\,n}
\xi_i^{p_{j,i}}
\quad j=1,2,\dotsc,\ell,
\end{equation}
where $\xi=(\xi_1,\xi_2,\dotsc,\xi_{2\,n})\in\mathbb{R}^{2\,n}$.

  \item (Find $S_0(\xi)$) Use Fa\`a di Bruno's formula in \thref{thm: Faa di Bruno's formula use to find general formula for S0} with $g(r)=r^{\alpha}$ and $f=u(.,t)$, and the equation $u_{\hat{t}}=u_{xx}$ to find $S_0(\xi)$ determined by
\begin{equation}\label{eqn: S0 in step in algorithm}
(\alpha-1)\dfrac{d^n}{d\hat{t}^n}H_{\alpha}[u(2\hat{t})]
=
-(\alpha-1)\int_\mathbb{R}\dfrac{\partial^n}{\partial \hat{t}^n}u^{\alpha}\,dx
=
-\alpha\,(\alpha-1)\int_{\mathbb{R}} u^{\alpha}S_0(\xi)\,dx,
\end{equation}    
where 
\begin{equation}
S_0(\xi)=\frac{1}{\alpha\,u^{\alpha}}\dfrac{\partial^n}{\partial \hat{t}^n}u^{\alpha}.
\end{equation}

  \item (Integration by parts formulas) Employ \thref{lem: general formula for ibp in algorithm} to find \textit{all possible} integration by parts formulas
\begin{equation}\label{eqn: IBP formulae in step in algorithm}
\int_{\mathbb{R}}u^{\alpha}\,T_j(\xi)\,dx=0, \quad j=1,2,\cdots,r,
\end{equation}
where $k$ is the number of solutions to the equation 
\begin{equation}\label{eqn: eqn for mi index eqn}
\sum_{i=1}^{2\,n-1}i \cdot p_{i}=2\,n-1, \quad p_i\in \mathbb{N}\cup\{0\}. 
\end{equation}
  \item (Alternative representation of time derivatives) Use \eqref{eqn: IBP formulae in step in algorithm} to rewrite \eqref{eqn: S0 in step in algorithm} as
  
\begin{equation}\label{eqn: Halpha=So=S0+T in algorithm}
(\alpha-1)\dfrac{d^n}{d\hat{t}^n}H_{\alpha}[u(2\hat{t})]
=
-\alpha\,(\alpha-1)\int_{\mathbb{R}} u^{\alpha}S_0(\xi)\,dx
=-\alpha\,(\alpha-1)\int_{\mathbb{R}} u^{\alpha}\left((S_0+\sum_{i=1}^{r} c_i\,T_i)(\xi)\right)\,dx,
\end{equation}       
\textit{for any} $c_i\in\mathbb{R}$ ($i=1,\dotsc,r$). 
\item (Positiveness of $S_{\alpha}(\xi)$) We use step (2) to form the problem of finding $c_i\in\mathbb{R}$ ($i=1,\cdots,r$) such that 
\begin{equation}
S_{\alpha}(\xi)
:=(-1)^{n}\left(S_0+\sum_{i=1}^{r} c_i\,T_i\right)(\xi)
=
\sum_{j=1}^{\ell} k_j\,\mathcal{Q}_j(\xi)
\ge0, \quad \forall \xi\in\mathbb{R}^{2\,n}.
\end{equation}
With $S_{\alpha}(\xi)$ given in the above equation, \eqref{eqn: Halpha=So=S0+T in algorithm} becomes
\begin{equation}\label{eqn: Halpha and Salpha with factor alpha-1}
(\alpha-1)\,\dfrac{d^n}{d\hat{t}^n}H_{\alpha}[u(2\hat{t})]
=(-1)^{n+1}\alpha\,(\alpha-1)\int_{\mathbb{R}} u^{\alpha}
\,S_{\alpha}(\xi)
\,dx.
\end{equation}
$S_{\alpha}(\xi)$ is of degree $2\,n$ since each time derivative means using $u_{\hat{t}}=u_{xx}$, and this results in differentiation with respect to $x$ twice.
\item (Zero coefficients in $S_{\alpha}(\xi)$) The condition of the positiveness of $S_{\alpha}(\xi)$ \textit{for all} $\xi\in\mathbb{R}^{2\,n}$ leads to some vanishing coefficients in $S_{\alpha}(\xi)$, and thus certain $c_i's$ are determined.
\item (Determine $\alpha$ such that $S_{\alpha}(\xi)\ge0$, $\forall \xi\in\mathbb{R}^{2\,n}$) For different $n's$, we use different approaches. 



\begin{enumerate}[(a)]
  \item $n=2$ $\implies$ use \thref{lem: discriminant for quadratic} to determine $\alpha$ and the the remaining $c_i's$. 
  \item $n=3$ $\implies$ use \thref{lem: discriminant for the 6th degree poly corresp. 3rd time derivative} to determine $\alpha$ and the the remaining $c_i's$.
  \item $n=4,5$ $\implies$ \thref{thm: positive semi-definite <=>eigenvalues>=0} as follows is employed to determine $\alpha$ and the the remaining $c_i's$.

\begin{theorem}[Theorem 9.1. in \cite{xia2016automated}]\thlabel{thm: positive semi-definite <=>eigenvalues>=0}
\ \

A real symmetric matrix $M$ is positive semi-definite if and only if one of the following conditions holds.
\begin{enumerate}[(1)]
  \item All the roots of the characteristic polynomial of M are non-negative.
  \item There exists a real matrix $V$ such that $M=V\,V^T$.
  \item All the principal minors of $M$ are non-negative.
\end{enumerate}

\end{theorem}    
  
\end{enumerate}

\end{enumerate}

\vspace{5mm}
\section{Second-order time derivative}\label{sec: Second-order time derivative}
\vspace{5mm}


Following the algorithmic construction of entropies in Section~\ref{sec: Algorithmic construction of entropies}, we use \eqref{eqn: S0 in step in algorithm} with $n=2$ to find  
\begin{equation}\label{eqn: S0 in step 2nd time derivative}
(\alpha-1)\dfrac{d^2}{d\hat{t}^2}H_{\alpha}[u(2\hat{t})]
=
-(\alpha-1)\int_\mathbb{R}\dfrac{\partial^2}{\partial \hat{t}^2}u^{\alpha}\,dx
=
-\alpha\,(\alpha-1)\int_{\mathbb{R}} u^{\alpha}S_0(\xi)\,dx.
\end{equation}
Then we employ \thref{lem: general formula for ibp in algorithm} with $n=2$ to find \textit{all possible} integration by parts formulas in 

\vspace{5mm}

\begin{lemma}\thlabel{lem: IBP formula for 2nd time derivative}
\begin{equation}\label{eqn: IBP formulae 2nd derivative}
\int_{\mathbb{R}}u^{\alpha}\,T_j(\xi)\,dx=0, \quad j=1,2,3,
\end{equation}
where
\begin{subequations}\label{eqn: def T1-T3 2nd time derivative}
\begin{eqnarray}
\label{eqn: def of T1 2nd time derivative}
T_1(\xi) & = &  (\alpha-1)\,\xi_1\,\xi_3+\xi_4, \\[2mm]
\label{eqn: def of T2 2nd time derivative}      
T_2(\xi) & = &  (\alpha-2)\,\xi _1^2\,\xi _2+\xi_1\,\xi_3+\xi _2^2, \\[2mm]
\label{eqn: def of T3 2nd time derivative}
T_3(\xi) & = & (\alpha -3)\,\xi _1^4 +3\,\xi _1^2\,\xi_2.
\end{eqnarray}
\end{subequations}
\end{lemma}

\vspace{5mm}

The integration by parts formulas \eqref{eqn: IBP formulae 2nd derivative} in \thref{lem: IBP formula for 2nd time derivative} are determined by means of \eqref{eqn: eqn for mi index eqn} with $n=2$:
\begin{equation}
\sum_{i=1}^{3}i \cdot p_{i}=3, \quad p_i\in \mathbb{N}\cup\{0\},
\end{equation}
which has the solutions
\begin{subequations}
\begin{eqnarray}
(p_1,p_2,p_3) & = & (0,0,1), \\ 
(p_1,p_2,p_3) & = & (1,1,0), \\    
(p_1,p_2,p_3) & = & (3,0,0).     
\end{eqnarray}
\end{subequations}
Now we use \thref{lem: IBP formula for 2nd time derivative} to rewrite \eqref{eqn: S0 in step 2nd time derivative} 
\begin{equation}\label{eqn: Halpha=So=S0+T 2nd time derivative}
(\alpha-1)\dfrac{d^2}{d\hat{t}^2}H_{\alpha}[u(2\hat{t})]
=
-\alpha\,(\alpha-1)\int_{\mathbb{R}} u^{\alpha}S_0(\xi)\,dx
=-\alpha\,(\alpha-1)\int_{\mathbb{R}} u^{\alpha}\left((S_0+\sum_{i=1}^{3} c_i\,T_i)(\xi)\right)\,dx
\end{equation}
\textit{for any} $c_i\in\mathbb{R}$ ($i=1,2,3$). Our goal is to find $c_i\in\mathbb{R}$ ($i=1,2,3$) such that 
\begin{equation}
S_{\alpha}(\xi)
=\left(S_0+\sum_{i=1}^{3} c_i\,T_i\right)(\xi)\ge0, \quad \forall \xi=(\xi_1,\xi_2,\xi_3,\xi_4)\in\mathbb{R}^{4}.
\end{equation}
With $S_{\alpha}(\xi)$ given in the above equation, \eqref{eqn: Halpha=So=S0+T 2nd time derivative} becomes
\begin{equation}\label{eqn: Halpha and Salpha with factor alpha-1 2nd time derivative}
(\alpha-1)\,\dfrac{d^2}{d\hat{t}^2}H_{\alpha}[u(2\hat{t})]
=-
\alpha\,(\alpha-1)\,\int_{\mathbb{R}} u^{\alpha}
\,S_{\alpha}(\xi)
\,dx.
\end{equation}
It turns out that 
\begin{equation}\label{eqn: defi of S alpha 2nd time derivative}
S_{\alpha}(\xi)
:=\left(S_0+\sum_{i=1}^{3} c_i\,T_i\right)(\xi)
=k_1\,\xi_1^4+k_2\,\xi_1^2\,\xi_2+k_3\,\xi_1\,\xi_3+k_4\,\xi_2^2+k_5\,\xi_4,
\end{equation}
where 
\begin{subequations}\label{eqn: def k1-k5 2nd time derivative}
\begin{eqnarray}
\label{eqn: def of k1 2nd time derivative}
k_1 & = & (\alpha -3)\,c_3, \\[2mm]
\label{eqn: def of k2 2nd time derivative}      
k_2 & = & (\alpha -2)\,c_2+3\,c_3, \\[2mm]
\label{eqn: def of k3 2nd time derivative}
k_3 & = & (\alpha -1)\,c_1+c_2-1,\\[2mm]
\label{eqn: def of k4 2nd time derivative}
k_4 & = & c_2+1, \\[2mm]
\label{eqn: def of k5 2nd time derivative}      
k_5 & = & c_1.
\end{eqnarray}
\end{subequations}
We need the following lemma, which is essentially identical to Lemma 11. in \cite{Jungel-algorithmic-construction-of-entropies-Nonlinerity-2006}. 

\vspace{5mm}

\begin{lemma}\thlabel{lem: discriminant for quadratic}
Let $S_{\alpha}(\xi)$ be given by \eqref{eqn: defi of S alpha 2nd time derivative}. Then $S_{\alpha}(\xi)\ge 0$ for $\xi=(\xi_1,\xi_2,\xi_3,\xi_4)\in\mathbb{R}^4$ if and only if
\begin{enumerate}[(1)]
\item  $k_3=k_5=0$, 
\end{enumerate}
and one of the following statements holds:
\begin{enumerate}[(2a)]
\item $k_4>0$ and $4\,k_1\,k_4-k_2^2\ge0$;
\item $k_2=k_4=0$ and $k_1\ge0$.
\end{enumerate}
\end{lemma}
\begin{proof}
(2a) and (2b) follows from \cite{Jungel-algorithmic-construction-of-entropies-Nonlinerity-2006}. From (2a) and (2b), it follows that $k_1\ge0$. We prove (1) as follows:
\begin{itemize}
  \item Since $S_{\alpha}(0,0,0,\xi_4)=k_5\,\xi _4$, we obtain $k_5=0$.
  \item When $k_5=0$, $S_{\alpha}(1,0,1,0)=k_1+k_3\,\xi_3$, which yields $k_3=0$.
\end{itemize}
This completes the proof.

\end{proof}
It follows immediately from \eqref{eqn: def of k3 2nd time derivative} and \eqref{eqn: def of k5 2nd time derivative} that $k_3=k_5=0$ in \thref{lem: discriminant for quadratic} leads to
\begin{subequations}\label{eqn: determine c1and c2}
\begin{eqnarray}
\label{eqn: c1 2nd time derivative}
c_1 & = & 0, \\[2mm]
\label{eqn: c2 2nd time derivative}      
c_2 & = & 1.
\end{eqnarray}
\end{subequations}
Under \eqref{eqn: determine c1and c2}, \eqref{eqn: defi of S alpha 2nd time derivative} becomes
\begin{equation}\label{eqn: defi of S alpha 2nd time derivative under c1=0 and c2=1}
S_{\alpha}(\xi)
=(\alpha -3)\,c_3\,\xi _1^4+\left(\alpha +3\,c_3-2\right)\,\xi _1^2\,\xi _2
   +2\,\xi _2^2.
\end{equation}
Using  (2a) and (2b) in \thref{lem: discriminant for quadratic}, we are led to 
\begin{equation}
\alpha\in(0,1)\cup(1,3].
\end{equation}

\begin{theorem}[Second-order time derivative]\thlabel{def: Second Time Derivative}
Let $u=u(x,t)$ be the solution of \eqref{eqn: Cauchy problem for heat eqn}. Suppose that $\alpha\in(0,1)\cup(1,3]$. Then for $t>0$,
$$
\dfrac{d^2}{dt^2}H_{\alpha}[u(t)]\le0.
$$

\end{theorem}

\vspace{5mm}

In particular, we find in \eqref{eqn: Halpha and Salpha with factor alpha-1 2nd time derivative}:
\begin{enumerate}[(a)]
  \item Letting $\alpha\to1$ and $c_3=-1$ in \eqref{eqn: defi of S alpha 2nd time derivative under c1=0 and c2=1} gives
\begin{equation}\label{}
\dfrac{d^2}{dt^2}H_{\alpha}[u(t)]
=
\frac{1}{4}\dfrac{d^2}{d\hat{t}^2}H_{\alpha}[u(2\hat{t})]
=-\frac{1}{2}\int_{\mathbb{R}} u
\,
\left(
\xi_1^2
-\,\xi _2
\right)^2
\,dx.
\end{equation}  

\item Letting $\alpha\to1$ and $c_3=-\frac{1}{9}$ in \eqref{eqn: defi of S alpha 2nd time derivative under c1=0 and c2=1} gives
\begin{equation}
\dfrac{d^2}{dt^2}H_{\alpha}[u(t)]
=
\frac{1}{4}\dfrac{d^2}{d\hat{t}^2}H_{\alpha}[u(2\hat{t})]
=-\frac{1}{18}\,\int_{\mathbb{R}} u
\,
\left(
\xi_1^2
-3\,\xi _2
\right)^2
\,dx.
\end{equation}

\item Letting $\alpha\to1$ and $c_3=-\frac{5}{9}$ in \eqref{eqn: defi of S alpha 2nd time derivative under c1=0 and c2=1} gives
\begin{equation}
\dfrac{d^2}{dt^2}H_{\alpha}[u(t)]
=
\frac{1}{4}\dfrac{d^2}{d\hat{t}^2}H_{\alpha}[u(2\hat{t})]
=-\frac{1}{4}\int_{\mathbb{R}} u
\,
\left(
\frac{10}{9}\,\left(\xi_1^2-\frac{6\,\xi_2}{5}\right)^2+\frac{2\,\xi_2^2}{5}
\right)
\,dx.
\end{equation}

\end{enumerate}

This shows that \thref{def: Second Time Derivative} generalizes the result in \cite{Villani-short-proof-concavity-entropy-power-IEEE-2000} in one dimension of space. The representation in (a) already exists in the literature. To the best of the author's knowledge, the representations in (b) and (c) seem not to exist in the literature.

\vspace{5mm}
\section{Third-order time derivative}
\vspace{5mm}

We first use \eqref{eqn: S0 in step in algorithm} with $n=3$ to find  
\begin{equation}\label{eqn: S0 in step 3rd time derivative}
(\alpha-1)\dfrac{d^3}{d\hat{t}^3}H_{\alpha}[u(2\hat{t})]
=
-(\alpha-1)\int_\mathbb{R}\dfrac{\partial^3}{\partial \hat{t}^3}u^{\alpha}\,dx
=
-\alpha\,(\alpha-1)\int_{\mathbb{R}} u^{\alpha}S_0(\xi)\,dx.
\end{equation}
Then we employ \thref{lem: general formula for ibp in algorithm} with $n=3$ to find \textit{all possible} integration by parts formulas in 

\vspace{5mm}

\begin{lemma}\thlabel{lem: IBP formula for 4th time derivative}
\begin{equation}
\int_{\mathbb{R}}u^{\alpha}\,T_j(\xi)\,dx=0, \quad j=1,2,\dotsc,7,
\end{equation}
where
\begin{subequations}\label{eqn: def T1-T7 3rd time derivative}
\begin{eqnarray}
\label{eqn: def of T1 3rd time derivative}
T_1(\xi) & = &  (\alpha  -1)\,\xi _1\,\xi _5+\xi _6, \\[2mm]
\label{eqn: def of T2 3rd time derivative}      
T_2(\xi) & = & (\alpha  -2)\,\xi _1\,\xi _2\,\xi _3 +\xi_3^2+\xi _2\,\xi _4, \\[2mm]
\label{eqn: def of T3 3rd time derivative}
T_3(\xi) & = & (\alpha  -2)\,\xi _1^2\,\xi _4 +\xi _1\,\xi_5+\xi _2 \xi _4,\\[2mm]
\label{eqn: def of T4 3rd time derivative}
T_4(\xi) & = & (\alpha  -3)\,\xi _1^2\,\xi _2^2 +\xi _2^3+2\,\xi _1\,\xi _2\,\xi _3, \\[2mm]
\label{eqn: def of T5 3rd time derivative}      
T_5(\xi) & = & (\alpha  -3)\,\xi _1^3\,\xi _3+\xi_1^2\,\xi _4+2\,\xi _1\,\xi _2\,\xi _3, \\[2mm]
\label{eqn: def of T6 3rd time derivative}
T_6(\xi) & = &  (\alpha  -4)\,\xi _1^4\,\xi _2 +\xi_1^3\,\xi _3+3\,\xi _1^2\,\xi _2^2,\\[2mm]
\label{eqn: def of T7 3rd time derivative}
T_7(\xi) & = & (\alpha  -5)\,\xi _1^6 +5\,\xi _1^4\,\xi _2. 
\end{eqnarray}
\end{subequations}

\end{lemma}

\vspace{5mm}

Now we use \thref{lem: IBP formula for 4th time derivative} to rewrite \eqref{eqn: S0 in step 3rd time derivative} as 
\begin{equation}\label{eqn: Halpha=So=S0+T 3rd time derivative}
(\alpha-1)\dfrac{d^3}{d\hat{t}^3}H_{\alpha}[u(2\hat{t})]
=
-\alpha\,(\alpha-1)\int_{\mathbb{R}} u^{\alpha}S_0(\xi)\,dx
=-\alpha\,(\alpha-1)\int_{\mathbb{R}} u^{\alpha}\left((S_0+\sum_{i=1}^{7} c_i\,T_i)(\xi)\right)\,dx,
\end{equation}
where $c_i\in\mathbb{R}$ ($i=1,\dotsc,7$). Our goal is to find $c_i\in\mathbb{R}$ ($i=1,\dotsc,7$) such that 
\begin{equation}
S_{\alpha}(\xi)
:=-\left(S_0+\sum_{i=1}^{7} c_i\,T_i\right)(\xi)\ge0, \quad \forall \xi=(\xi_1,\xi_2,\xi_3,\xi_4,\xi_5)\in\mathbb{R}^{5}.
\end{equation}
With $S_{\alpha}(\xi)$ given in the above equation, \eqref{eqn: Halpha=So=S0+T 3rd time derivative} becomes
\begin{equation}\label{eqn: Halpha and Salpha with factor alpha-1 3rd time derivative}
(\alpha-1)\,\dfrac{d^3}{d\hat{t}^3}H_{\alpha}[u(2\hat{t})]
=-
\alpha\,(\alpha-1)\,\int_{\mathbb{R}} u^{\alpha}
\,S_{\alpha}(\xi)
\,dx.
\end{equation}
It turns out that 
\begin{align}\label{eqn: defi of S alpha 3rd time derivative}
\notag
S_{\alpha}(\xi):=&-\left(S_0+\sum_{i=1}^{7} c_i\,T_i\right)(\xi)
\\[1mm] \notag
=& k_1\,\xi_1^6+k_2\,\xi_1^4\,\xi_2+k_3\,\xi_1^3\,\xi_3+k_4\,\xi_1^2\,\xi_2^2+k_5\,\xi_1\,\xi_2\,\xi_3+k_6\,\xi_3^2
\\[2mm] 
&+k_7\,\xi_2^3+k_8\,\xi_1^2\,\xi_4+k_9\,\xi_2\,\xi_4+k_{10}\,\xi_1\,\xi_5+k_{11}\,\xi_6,
\end{align}
where 
\begin{subequations}\label{eqn: def k1-k11 3rd time derivative}
\begin{eqnarray}
\label{eqn: def of k1}
k_1 & = &(\alpha -5)\,c_7 \\[2mm]
\label{eqn: def of k2}      
k_2 & = & (\alpha -4)\,c_6+5\,c_7, \\[2mm]
\label{eqn: def of k3}
k_3 & = & (\alpha -3)\,c_5+c_6,\\[2mm]
\label{eqn: def of k4}
k_4 & = & (\alpha -3)\,c_4+3\,c_6, \\[2mm]
\label{eqn: def of k5}      
k_5 & = & (\alpha -2)\,c_2+2\left(c_4+c_5\right), \\[2mm]
\label{eqn: def of k6}
k_6 & = & c_2,\\[2mm]
\label{eqn: def of k7}
k_7 & = & (\alpha -2) (\alpha -1)+c_4,\\[2mm]
\label{eqn: def of k8}
k_8 & = & (\alpha -2)\,c_3+c_5,\\[2mm]
\label{eqn: def of k9}
k_9 & = & 3 (\alpha -1)+c_2+c_3,\\[2mm]
\label{eqn: def of k10}
k_{10} & = & (\alpha -1)\,c_1+c_3,\\[2mm]
\label{eqn: def of k11}
k_{11} & = & c_1+1.
\end{eqnarray}
\end{subequations}
We need the following \thref{lem: discriminant for the 6th degree poly corresp. 3rd time derivative} in \cite{Jungel-algorithmic-construction-of-entropies-Nonlinerity-2006} to determine $c_i$ ($i=1,\cdots,7$).  

\vspace{5mm}

\begin{lemma}\thlabel{lem: discriminant for the 6th degree poly corresp. 3rd time derivative} Let $S_{\alpha}(\xi)$ be given by 
\begin{align}
\notag
S_{\alpha}(\xi)=& k_1\,\xi_1^6+k_2\,\xi_1^4\,\xi_2+k_3\,\xi_1^3\,\xi_3+k_4\,\xi_1^2\,\xi_2^2+k_5\,\xi_1\,\xi_2\,\xi_3+k_6\,\xi_3^2
\\[2mm] 
&+k_7\,\xi_2^3+k_8\,\xi_1^2\,\xi_4+k_9\,\xi_2\,\xi_4+k_{10}\,\xi_1\,\xi_5+k_{11}\,\xi_6.
\end{align}
Then $S_{\alpha}(\xi)\ge 0$ for $\xi=(\xi_1,\xi_2,\xi_3,\xi_4,\xi_5,\xi_6)\in\mathbb{R}^6$ if and only if
\begin{enumerate}[(1)]
\item  $k_7=k_8=k_9=k_{10}=k_{11}=0$, 
\end{enumerate}
and one of the following statements holds:
\begin{enumerate}[(2a)]
\item $4\,k_4\,k_6-k_5^2>0$ and $4\,k_1\,k_4\,k_6-k_1\,k_5^2-k_2^2\,k_6-k_3^2\,k_4+k_2\,k_3\,k_5\ge0$.
\item $4\,k_4\,k_6-k_5^2=2\,k_2\,k_6-k_3\,k_5=0$ and $4\,k_1\,k_6-k_3^2\ge0$.
\end{enumerate}
\end{lemma}

\begin{proof}
(2a) and (2b) follows from \cite{Jungel-algorithmic-construction-of-entropies-Nonlinerity-2006}. From (2a) and (2b), it follows that $k_1\ge0$. We prove (1) as follows:
\begin{itemize}
  \item Since $S_{\alpha}(0,\xi_2,0,0,\xi_5,\xi_6)=k_7\,\xi _2^3+k_{11}\,\xi _6$, we obtain $k_7=k_{11}=0$.
  \item When $k_7=k_{11}=0$, $S_{\alpha}(1,0,0,0,\xi_5,\xi_6)=k_1+k_{10}\,\xi_5$, which yields $k_{10}=0$.
  \item When $k_7=k_{10}=k_{11}=0$, $S_{\alpha}(1,0,0,\xi_4,\xi_5,\xi_6)=k_1+k_8\,\xi_4$. This gives $k_8=0$.
  \item When $k_7=k_8=k_{10}=k_{11}=0$, $S_{\alpha}(0,\xi_2,0,\xi_4,\xi_5,\xi_6)=k_9\,\xi_2\,\xi_4$. We show that $k_9=0$.
\end{itemize}
This completes the proof.
\end{proof}
We use (1) in \thref{lem: discriminant for the 6th degree poly corresp. 3rd time derivative} to determine $c_i$ ($i=1,\dotsc,5$) as follows:
\begin{subequations}\label{eqn: c1 to c7 determine 3rd time derivtive}
\begin{eqnarray}
\label{eqn: c1 3rd time derivative}
c_1 & = & 0, \\[2mm]
\label{eqn: c2 3rd time derivative}
c_2 & = & 4, \\[2mm]
\label{eqn: c3 3rd time derivative}
c_3 & = & -1,\\[2mm]
\label{eqn: c4 3rd time derivative}
c_4 & = & \alpha-2,\\[2mm]
\label{eqn: c5 3rd time derivative}
c_5 & = & \alpha-2.
\end{eqnarray}
\end{subequations}
Using \eqref{eqn: c1 to c7 determine 3rd time derivtive}, \eqref{eqn: defi of S alpha 3rd time derivative} becomes
\begin{align}\label{eqn: defi of S alpha 3rd time derivative under some ci's}
S_{\alpha}(\xi)=&(\alpha -5)\,c_7\,\xi _1^6+\left((\alpha -4)\,c_6+5\,c_7\right)\,\xi _1^4\,\xi _2 + \left((\alpha-2)\,(\alpha-3) +c_6\right)\,\xi_1^3\,\xi _3
\\[1mm] \notag
&+\left((\alpha-2)\,(\alpha-3) +3\,c_6\right)\,\xi _1^2\,\xi _2^2+8\,(\alpha -2)\,\xi _1\,\xi _2\,\xi _3+4 \,\xi _3^2.
\end{align}
Using  (2a) and (2b) in \thref{lem: discriminant for the 6th degree poly corresp. 3rd time derivative} , we are led to
\begin{equation}\label{eqn: set of alpha 3rd time derivative}
\alpha\in(\alpha_0,1)\cup(1,2],
\end{equation}
where $\alpha_0\approx0.389214$ is the real root of the cubic equation $9 \alpha^3-12 \alpha^2+29 \alpha-10=0$.

\vspace{5mm}

\begin{theorem}[Third-order time derivative]\thlabel{def: Third Time Derivative}
Let $u=u(x,t)$ be the solution of \eqref{eqn: Cauchy problem for heat eqn}. Suppose that $\alpha\in(\alpha_0,1)\cup(1,2]$, where $\alpha_0\approx0.389214$ is the real root of the cubic equation $9 \alpha^3-12 \alpha^2+29 \alpha-10=0$. Then for $t>0$,
$$
\dfrac{d^3}{dt^3}H_{\alpha}[u(t)]\ge0.
$$

\end{theorem}

\vspace{5mm}

In particular, we find in \eqref{eqn: Halpha and Salpha with factor alpha-1 3rd time derivative}:
\begin{enumerate}[(a)]
  \item Letting $\alpha\to1$ and $(c_6,c_7)=(\frac{2}{3},\frac{-2}{15})$ in \eqref{eqn: defi of S alpha 3rd time derivative under some ci's} gives
\begin{equation}\label{}
\dfrac{d^3}{dt^3}H_{\alpha}[u(t)]
=
\frac{1}{8}\dfrac{d^3}{d\hat{t}^3}H_{\alpha}[u(2\hat{t})]
=\frac{1}{2}\int_{\mathbb{R}} u\,
\left(
\left(\frac{1}{3}\,\xi_1^3-\xi_1\,\xi_2+\xi_3\right)^2
+\frac{1}{45}\xi_1^6
\right)
\,dx,
\end{equation}  
which recovers Theorem 1. in \cite{Cheng-Higher-order-derivatives-Costa-IEEE-2015}.

\item Letting $\alpha\to1$ and $(c_6,c_7)=(\frac{1283}{1102},\frac{-39}{80})$ in \eqref{eqn: defi of S alpha 3rd time derivative under some ci's} gives another representation as follows: 
   
\begin{align}
\dfrac{d^3}{dt^3}H_{\alpha}[u(t)]
=
\frac{1}{8}\dfrac{d^3}{d\hat{t}^3}H_{\alpha}[u(2\hat{t})]
=
&\frac{1}{8}\int_{\mathbb{R}} u\,
\Bigg(
\left(\frac{1}{2} \sqrt{\frac{39}{5}}\,\xi
   _1^3-\frac{17427}{8816}\sqrt{\frac{15}{13}}\,\xi _1
   \xi _2+\frac{3487}{1102}\sqrt{\frac{5}{39}}\,
   \xi _3\right)^2
\\[1mm] \notag
&+\left(\frac{\sqrt{
   994272857}}{8816\sqrt{13}}\,\xi _1 \xi
   _2-\frac{201352319}{1102
   \sqrt{12925547141}}\,\xi _3\right)^2
\\[1mm] \notag   
&+\frac{219395023060}{1643533032621}\,\xi
   _3^2
   \Bigg)\,dx.
\end{align}

\end{enumerate}

\vspace{5mm}
\section{Fourth-order time derivative}\label{sec: Fourth-order time derivative}
\vspace{5mm}

We first use \eqref{eqn: S0 in step in algorithm} with $n=4$ to find  
\begin{equation}\label{eqn: S0 in step 4th time derivative}
(\alpha-1)\dfrac{d^4}{d\hat{t}^4}H_{\alpha}[u(2\hat{t})]
=
-(\alpha-1)\int_\mathbb{R}\dfrac{\partial^4}{\partial \hat{t}^4}u^{\alpha}\,dx
=
-\alpha\,(\alpha-1)\int_{\mathbb{R}} u^{\alpha}S_0(\xi)\,dx.
\end{equation}
Then we employ \thref{lem: general formula for ibp in algorithm} with $n=4$ to find \textit{all possible} integration by parts formulas in 

\vspace{5mm}

\begin{lemma}\thlabel{lem: IBP formula for 4th time derivative 2}
\begin{equation}\label{eqn: IBP formulae 4th derivative}
\int_{\mathbb{R}}u^{\alpha}\,T_j(\xi)\,dx=0, \quad j=1,2,\dotsc,15,
\end{equation}
where 
\begin{subequations}\label{eqn: def of T1 to T15 4th time derivative}
\begin{eqnarray}
\label{eqn: def of T1 4th time derivative}
T_1(\xi) & = & (\alpha  -1)\,\xi _1\,\xi _7 +\xi _8, \\[2mm]
\label{eqn: def of T2 4th time derivative}
T_2(\xi) & = &  (\alpha  -2)\,\xi _1\,\xi _3\,\xi _4+\xi _4^2+\xi _3\,\xi _5, \\[2mm]
\label{eqn: def of T3 4th time derivative}
T_3(\xi) & = &  (\alpha  -2)\,\xi _1\,\xi _2\,\xi _5+\xi _3\,\xi _5+\xi _2\,\xi _6, \\[2mm]
\label{eqn: def of T4 4th time derivative}
T_4(\xi) & = & (\alpha  -3)\,\xi _1\,\xi _3\,\xi _2^2+\xi _4\,\xi _2^2+2\,\xi _3^2 \xi _2, \\[2mm]
\label{eqn: def of T5 4th time derivative}
T_5(\xi) & = & (\alpha  -2)\,\xi _6\,\xi _1^2 +\xi _7\,\xi _1+\xi _2\,\xi _6, \\[2mm]
\label{eqn: def of T6 4th time derivative}
T_6(\xi) & = & (\alpha  -3)\,\xi _1^2\,\xi _3^2 +\xi _2\,\xi _3^2+2 \xi _1\,\xi _4\,\xi _3, \\[2mm]
\label{eqn: def of T7 4th time derivative}
T_7(\xi) & = & (\alpha  -3)\,\xi _2\,\xi _4\,\xi _1^2 +\xi _3\,\xi _4\,\xi _1+\xi _2 \,\xi _5\,\xi _1+\xi _2^2\,\xi _4, \\[2mm]
\label{eqn: def of T8 4th time derivative}
T_8(\xi) & = &(\alpha  -4)\,\xi _1^2\,\xi _2^3 +\xi _2^4+3\,\xi _1\,\xi _3\,\xi _2^2, \\[2mm]
\label{eqn: def of T9 4th time derivative}
T_9(\xi) & = & (\alpha  -3)\,\xi _5\,\xi _1^3 +\xi _6\,\xi _1^2+2\,\xi _2\,\xi _5\,\xi _1, \\[2mm]
\label{eqn: def of T10 4th time derivative}
T_{10}(\xi) & = & (\alpha  -4)\,\xi _2\,\xi _3\,\xi _1^3 +\xi _3^2\,\xi _1^2+\xi _2\, \xi _4\,\xi_1^2+2\,\xi _2^2\,\xi _3\,\xi _1, \\[2mm]
\label{eqn: def of T11 4th time derivative}
T_{11}(\xi) & = & (\alpha  -4)\,\xi _4\,\xi _1^4 +\xi _5\,\xi _1^3+3\,\xi _2\,\xi _4 \,\xi _1^2, \\[2mm]
\label{eqn: def of T12 4th time derivative}
T_{12}(\xi) & = & (\alpha  -5)\,\xi _2^2\,\xi _1^4+2\,\xi _2\,\xi _3\,\xi _1^3+3\,\xi _2^3\,\xi_1^2, \\[2mm]
\label{eqn: def of T13 4th time derivative}
T_{13}(\xi) & = & (\alpha  -5)\,\xi _3\,\xi _1^5 +\xi _4\,\xi _1^4+4\,\xi _2\,\xi _3\, \xi _1^3, \\[2mm]
\label{eqn: def of T14 4th time derivative}
T_{14}(\xi) & = & (\alpha  -6)\,\xi _2\,\xi _1^6 +\xi _3\,\xi _1^5+5\,\xi _2^2\,\xi _1^4, \\[2mm]
\label{eqn: def of T15 4th time derivative}
T_{15}(\xi) & = & (\alpha  -7)\,\xi _1^8 +7\,\xi _2\,\xi _1^6.
\end{eqnarray}
\end{subequations}


\end{lemma}

\vspace{5mm}

Now we use \thref{lem: IBP formula for 4th time derivative} to rewrite \eqref{eqn: S0 in step 4th time derivative} as 
\begin{equation}\label{eqn: Halpha=So=S0+T 4th time derivative}
(\alpha-1)\dfrac{d^4}{d\hat{t}^4}H_{\alpha}[u(2\hat{t})]
=
-\alpha\,(\alpha-1)\int_{\mathbb{R}} u^{\alpha}S_0(\xi)\,dx
=-\alpha\,(\alpha-1)\int_{\mathbb{R}} u^{\alpha}\left((S_0+\sum_{i=1}^{15} c_i\,T_i)(\xi)\right)\,dx,
\end{equation}
where $c_i\in\mathbb{R}$ ($i=1,\dotsc,15$). Our goal is to find $c_i\in\mathbb{R}$ ($i=1,\dotsc,15$) such that 
\begin{equation}
S_{\alpha}(\xi)
:=-\left(S_0+\sum_{i=1}^{15} c_i\,T_i\right)(\xi)\ge0, \quad \forall \xi=(\xi_1,\xi_2,\xi_3,\xi_4,\xi_5,\xi_6,\xi_7,\xi_8)\in\mathbb{R}^{8}.
\end{equation}
With $S_{\alpha}(\xi)$ given in the above equation, \eqref{eqn: Halpha=So=S0+T 4th time derivative} becomes
\begin{equation}\label{eqn: Halpha and Salpha with factor alpha-1 4th time derivative}
(\alpha-1)\,\dfrac{d^4}{d\hat{t}^4}H_{\alpha}[u(2\hat{t})]
=-
\alpha\,(\alpha-1)\,\int_{\mathbb{R}} u^{\alpha}
\,S_{\alpha}(\xi)
\,dx.
\end{equation}
It turns out that 
\begin{align}\label{eqn: defi of S alpha 4th time derivative}
S_{\alpha}(\xi):=&\left(S_0+\sum_{i=1}^{15} c_i\,T_i\right)(\xi)
\\[1mm] \notag
=& \sum_{j=1}^{22} k_j\,\mathcal{Q}_j(\xi),
\end{align}
where $\xi=(\xi_1,\xi_2,\dotsc,\xi_{8})\in\mathbb{R}^{8}$ and

\begin{equation}
\mathcal{Q}_j(\xi):=
\prod_{i=1}^{8}
\xi_i^{p_{j,i}}
\quad j=1,2,\dotsc,22.
\end{equation}
In the above equation, for $j=1,2,\dotsc,22$, $\mathcal{P}_j=(p_{j,1},p_{j,2},\dotsc,p_{j,2\,m})$ is a solution of 
\begin{equation}
\sum_{i=1}^{8}i \cdot p_{j,i}=8, \quad p_{j,i}\in \mathbb{N}\cup\{0\}.
\end{equation}

We can show in the same manner as the proof of (1) in \thref{lem: discriminant for the 6th degree poly corresp. 3rd time derivative} that the coefficients of each of   
\begin{equation}
\xi_{8},\quad \xi_{3}\,\xi_{5}, \quad  \xi_{2}\,\xi_{6}, \quad \xi_{2}\,\xi_{3}^2, \quad \xi_{1}\,\xi_{7}, \quad  \xi_{1}\,\xi_{2}\,\xi_{5}, \quad \xi_{1}^2\,\xi_{6}, \quad \xi_{1}^3\,\xi_{5}
\end{equation}
in $S_{\alpha}(\xi)=\sum_{j=1}^{22} k_j\,\mathcal{Q}_j(\xi)$ must vanish. This yields 
\begin{subequations}\label{eqn: ci's determine 4th time derivtive}
\begin{eqnarray}
\label{eqn: c1 4th time derivative}
c_1 & = & 0, \\[2mm]
\label{eqn: c2 4th time derivative}
c_2 & = & 5, \\[2mm]
\label{eqn: c3 4th time derivative}
c_3 & = & -5,\\[2mm]
\label{eqn: c5 4th time derivative}
c_5 & = & 1,\\[2mm]
\label{eqn: c6 4th time derivative}
c_6 & = & -2\,c_4,\\[2mm] 
\label{eqn: c7 4th time derivative}
c_7 & = & 7\,(\alpha-2),\\[2mm] 
\label{eqn: c9 4th time derivative}
c_9 & = & 2-\alpha,\\[2mm] 
\label{eqn: c11 4th time derivative}
c_{11} & = & (\alpha-2)\,(\alpha-3).
\end{eqnarray}
\end{subequations}
Under the condition \eqref{eqn: ci's determine 4th time derivtive}, \eqref{eqn: defi of S alpha 4th time derivative} becomes
\begin{align}\label{eqn: defi of S alpha 4th time derivative under some ci's}
S_{\alpha}(\xi)
=&
(\alpha -7)\,c_{15}\,\xi _1^8+\left((\alpha-6)\,c_{14}+7\,c_{15}\right)\,\xi _1^6\,\xi _2+\left((\alpha-5)\,c_{12}+5\,c_{14}\right)\,\xi _1^4\,\xi _2^2  
\\[1mm] \notag
&+
\left((\alpha-2)\,(\alpha-3)+c_8\right)\,\xi _2^4 +\left((\alpha -5)\,c_{13}+c_{14}\right)\,\xi _1^5\,\xi _3  +\left((\alpha-4)\,c_8+3 c_{12}\right)\,\xi_1^2\,\xi _2^3  
\\[1mm] \notag
&+
\left((\alpha-4)\,c_{10}+2\,c_{12}+4\,c_{13}\right)\,\xi _1^3\,\xi _2\,\xi _3  +\left(2\,c_4\,(3-\alpha)+c_{10}\right)\, \xi_1^2 \,\xi _3^2
\\[1mm] \notag
&+
\left((\alpha-3)\,c_4+3\,c_8+2\,c_{10}\right)\,\xi _1\,\xi _2^2\,\xi _3+\left((\alpha-2)\,(\alpha-3)\,(\alpha-4)+c_{13}\right)\,\xi _1^4\,\xi _4  
\\[1mm] \notag
&+
\left(10\,(\alpha-2)\,(\alpha-3)+c_{10}\right)\,\xi_1^2\,\xi _2\,\xi _4+\left(13\,(\alpha-2) +c_4\right)\, \xi _2^2 \,\xi _4
\\[1mm] \notag
&+
4\left(3\,(\alpha-2)-c_4\right)\xi _1\,\xi _3\,\xi _4 +8\,\xi _4^2. 
\end{align}
We can rewrite \eqref{eqn: defi of S alpha 4th time derivative under some ci's} in the matrix form
\begin{equation}
S_{\alpha}(\xi)
=
\begin{bmatrix}
\xi_1^4 & \xi_1^2\,\xi_2  &  \xi_1\,\xi_3 &  \xi_2^2 &  \xi_4 \\
\end{bmatrix}
M
\begin{bmatrix}
\xi_1^4\\
\xi_1^2\,\xi_2\\
\xi_1\,\xi_3\\
\xi_2^2\\
\xi_4\\
\end{bmatrix},
\end{equation}
where 
\begin{equation}
M=
\begin{bmatrix}
b_{11} & b_{12}  &  b_{13} &  b_{14} &  b_{15} \\
b_{12} & b_{22}  &  b_{23} &  b_{24} &  b_{25} \\
b_{13} & b_{23}  &  b_{33} &  b_{34} &  b_{35} \\
b_{14} & b_{24}  &  b_{34} &  b_{44} &  b_{45} \\
b_{15} & b_{25}  &  b_{35} &  b_{45} &  b_{55} \\
\end{bmatrix}
\end{equation}
is a symmetric $5\times5$ matrix. Employing \thref{thm: positive semi-definite <=>eigenvalues>=0}, we use the Mathematica commands \verb"NMaximize" and \verb"Eigenvalues" to numerically find the condition $\alpha\in(1,2)$ which ensures the following 

\vspace{5mm}

\begin{theorem}[Fourth-order time derivative]\thlabel{def: 4th Time Derivative}
Let $u=u(x,t)$ be the solution of \eqref{eqn: Cauchy problem for heat eqn}. Suppose that $\alpha\in(1,2)$. Then for $t>0$,
$$
\dfrac{d^4}{dt^4}H_{\alpha}[u(t)]\le0.
$$

\end{theorem}


\vspace{5mm}

In particular, we find in \eqref{eqn: Halpha and Salpha with factor alpha-1 4th time derivative}:
\begin{enumerate}[(a)]
  \item Letting $\alpha\to1$ and $(c_{4},c_{8},c_{10},c_{12},c_{13},c_{14},c_{15})=(\frac{9}{5},\frac{146}{75},\frac{28}{5},-\frac{302}{75},-2,\frac{272}{125},-\frac{1516}{4375})$ in \eqref{eqn: defi of S alpha 4th time derivative under some ci's} gives
\begin{align}
\dfrac{d^4}{dt^4}H_{\alpha}[u(t)]
=
\frac{1}{16}\dfrac{d^4}{d\hat{t}^4}H_{\alpha}[u(2\hat{t})]
=&
-\frac{1}{2}\int_{\mathbb{R}} u\,
\Bigg(
\left(-\frac{1}{2}\xi _1^4+\frac{8}{5}\xi
   _1^2 \xi _2 -\frac{6}{5}\xi _1\xi _3 -\frac{7}{10} \xi
   _2^2+\xi _4\right)^2
\\[1mm] \notag
&+
\left(\frac{9}{100}\xi _1^4-\frac{1}{3} \xi _1^2\xi _2
   +\frac{2 }{5}\xi _1 \xi _2\right)^2
\\[1mm] \notag   
&+
\left(\frac{1}{25}\xi _1^4-\frac{1}{25} \xi _1^2
   \xi _2\right)^2
\\[1mm] \notag   
&+
\frac{13}{70000}\xi _1^8+\frac{7}{11250}\xi
   _1^4\xi _2^2 +\frac{1}{300}\xi _2^4
\Bigg)\,dx   
\end{align}  
which recovers Theorem 2. in \cite{Cheng-Higher-order-derivatives-Costa-IEEE-2015}.

\item Letting $\alpha\to1$ and $(c_{4},c_{8},c_{10},c_{12},c_{13},c_{14},c_{15})=(\frac{9}{4},\frac{17}{10},7,-\frac{23}{5},-\frac{5}{2},\frac{8}{3},-\frac{4}{9})$ in \eqref{eqn: defi of S alpha 4th time derivative under some ci's} gives an alternative representation as follows: 
   
\begin{align}
&\dfrac{d^4}{dt^4}H_{\alpha}[u(t)]
=
\frac{1}{16}\dfrac{d^4}{d\hat{t}^4}H_{\alpha}[u(2\hat{t})]
\\[1mm] \notag
&=
-\frac{1}{16}\int_{\mathbb{R}} u\,
\Bigg(
\left(2 \sqrt{\frac{2}{3}} \xi _1^4-\frac{37 }{3 \sqrt{6}}\xi _1^2\xi
   _2 +\frac{19 }{2 \sqrt{6}}\xi
   _1\xi _3 +\frac{3}{2}
   \sqrt{\frac{3}{2}} \xi _2^2-\frac{17}{8}
   \sqrt{\frac{3}{2}} \xi _4\right)^2
\\[1mm] \notag
&\hspace{17mm}+
\left(\frac{1}{3} \sqrt{\frac{103}{30}}  \xi _1^2\xi _2
  -\frac{1}{2} \sqrt{\frac{103}{30}}
   \xi _1\xi _3 -3 \sqrt{\frac{6}{515}} \xi
   _2^2+\frac{19}{8} \sqrt{\frac{15}{206}} \xi
   _4\right)^2
\\[1mm] \notag   
&\hspace{17mm}+
\left(-\frac{1}{4} \sqrt{\frac{5}{2}} \xi
   _2^2+\frac{1}{\sqrt{10}}\xi _1 \xi
   _3+\frac{3}{8} \sqrt{\frac{5}{2}}
   \xi _4\right)^2
\\[1mm] \notag   
&\hspace{17mm}+
\left(\frac{9}{4}
   \sqrt{\frac{13}{1030}} \xi _2^2-\frac{77}{72}
   \sqrt{\frac{65}{206}} \xi _4\right)^2+\frac{67}{648} \xi _4^2
\Bigg)\,dx.   
\end{align}

\end{enumerate}

\vspace{5mm}
\section{Fifth-order time derivative}
\vspace{5mm}

For the fifth-order time derivative of the Tsallis entropy, we mimic the proof of \thref{def: 4th Time Derivative} in Section~\ref{sec: Fourth-order time derivative}. It is easy to see that we need only to revise the following three parameters in the algorithm for constructing entropies (See Section~\ref{sec: Algorithmic construction of entropies})

\begin{itemize}
  \item $n=5$ (the order of derivative) The number of components in $\xi=(\xi_1,\xi_2,\dotsc,\xi_{2\,n})\in\mathbb{R}^{2\,n}$ is $2\,n=10$.
  \item $\ell=42$ (the number of $\mathcal{Q}_j(\xi)$ in $S_{\alpha}(\xi):=\displaystyle\sum_{j=1}^{\ell} k_j\,\mathcal{Q}_j(\xi)$) See step (2) in the algorithm introduced in Section~\ref{sec: Algorithmic construction of entropies}.

  \item $r=30$ (the number of all possible integration by parts formulas) See \thref{lem: general formula for ibp in algorithm}.
\end{itemize}

\vspace{5mm}


\begin{theorem}[Fifth-order time derivative]\thlabel{def: 5th Time Derivative}
Let $u=u(x,t)$ be the solution of \eqref{eqn: Cauchy problem for heat eqn}. Suppose that $\alpha\in(\alpha_0,2)$, where $\alpha_0\in(1.54,1.55)$. Then for $t>0$,
$$
\dfrac{d^5}{dt^5}H_{\alpha}[u(t)]\ge0.
$$

\end{theorem}

\vspace{5mm}

\begin{example}[SOS for the fifth order derivative of the Tsallis entropy along the heat flow]\thlabel{thm: SOS for the fifth order derivative of the Tsallis entropy along the heat flow} When $\alpha=\frac{9}{5}$, we can employ \thref{def: 5th Time Derivative} to obtain



\begin{equation}
\dfrac{d^5}{dt^5}H_{\alpha}[u(t)]
=
\frac{\alpha}{32}
\int_{\mathbb{R}} 
u^{\alpha}
\,S_{\alpha}(\xi)
\,dx,
\end{equation}
where

\begin{align}
\notag
S_{\alpha}(\xi)
=&
\left(
2 \sqrt{\frac{3}{5}} \xi _1^5-\frac{239}{92}
   \sqrt{\frac{3}{5}} \xi _2 \xi _1^3-\frac{5939
   \sqrt{\frac{5}{3}} \xi _3 \xi _1^2}{11776}+\frac{2 \xi
   _2^2 \xi _1}{\sqrt{15}}-\frac{46199 \xi _4 \xi
   _1}{25600 \sqrt{15}}+\frac{9}{20} \sqrt{\frac{3}{5}}
   \xi _2 \xi _3+\frac{2687 \xi _5}{1000 \sqrt{15}}
\right)^2
\\[2mm] \notag
+&
\Bigg(
\frac{1}{92} \sqrt{\frac{25517}{5}} \xi _2 \xi
   _1^3-\frac{18290909 \xi _3 \xi _1^2}{58880
   \sqrt{127585}}-\frac{13534 \xi _2^2 \xi _1}{25
   \sqrt{127585}}+\frac{12840167 \xi _4 \xi _1}{25600
   \sqrt{127585}}
\\[2mm] \notag
+&
\frac{661}{20} \sqrt{\frac{17}{7505}}
   \xi _2 \xi _3-\frac{1312807 \xi _5}{1000
   \sqrt{127585}}
\Bigg)^2
\\[2mm] \notag
+&
\Bigg(
\frac{1}{320} \sqrt{\frac{576677843801}{8803365}} \xi _3
   \xi _1^2-\frac{2277524879}{50}
   \sqrt{\frac{23}{220726328104051755}} \xi _2^2 \xi
   _1
\\[2mm] \notag
+&
\frac{742659980941
   \sqrt{\frac{23}{220726328104051755}} \xi _4 \xi
   _1}{25600}-\frac{14756469}{100}
   \sqrt{\frac{1173}{4327967217726505}} \xi _2 \xi
   _3
\\[2mm] \notag
+&
\frac{4907107601}{250}
   \sqrt{\frac{23}{220726328104051755}} \xi _5
\Bigg)^2
\\[2mm] \notag
+&
\Bigg(
\frac{1}{50} \sqrt{\frac{417614607411981}{576677843801}}
   \xi _1 \xi _2^2-\frac{5272521222948383 \xi _3 \xi
   _2}{500
   \sqrt{240829091342142315973979781}}
\\[2mm] \notag
-&
\frac{926095007426
   14817 \xi _1 \xi _4}{25600
   \sqrt{240829091342142315973979781}}-\frac{117323370584
   2694 \xi _5}{125 \sqrt{240829091342142315973979781}}
\Bigg)^2
\\[2mm] \notag
+&
\Bigg(
-\frac{94886586035603628733}{25}
   \sqrt{\frac{2}{15294622779318478268097
   0968262398165595}} \xi _2 \xi
   _3
\\[2mm] \notag
+&
\frac{\sqrt{\frac{73247546938558973
   587099}{4176146074119810}} \xi _1 \xi
   _4}{12800}+\frac{125522014411936570972
   9 \xi _5}{250
   \sqrt{30589245558636956536194193652479
   6331190}}
\Bigg)^2
\\[2mm] \notag
+&
\Bigg(
\frac{1}{250}
   \sqrt{\frac{112238553261609866098264097}{14649509
   3877117947174198}} \xi _2 \xi
   _3-\frac{62519425547411619227833381}{25}
\\[2mm] \notag
\cdot&
   \sqrt{\frac{2}{8221198698345720047203938598348120
   739517134084603}} \xi _5
\Bigg)^2
\\[2mm] \notag
+&
\frac{7904729769252907453736008579 \xi
   _5^2}{14029819157701233262283012125000}
\end{align}

\end{example}

\end{document}